\documentclass[]{amsart}       
\usepackage{amsfonts,amssymb,amsmath, amsthm, tikz}
\usepackage{color}
\usepackage[margin=1.5in]{geometry}

\theoremstyle{plain}
\newtheorem{theorem}{Theorem}
\newtheorem*{thm}{Theorem}

\newtheorem{lemma}[theorem]{Lemma}

\theoremstyle{definition}

\newtheorem{definition}{Definition}

\newtheorem*{rem}{Remark}

\newcommand{\di}{\operatorname{div}}

\allowdisplaybreaks

\definecolor{myblue}{cmyk}{1,0.7,0,0}
\definecolor{mylightblue}{cmyk}{0.5,0.3,0,0}
\definecolor{mypaleblue}{cmyk}{0.08,0.05,0,0}
\definecolor{mypurple}{cmyk}{0.6,1,0,0}
\definecolor{myorange}{cmyk}{0,0.83,1,0.5}
\definecolor{mygreen}{cmyk}{0.9,0,1,0.2}
\definecolor{myred}{cmyk}{0,1,1,0.1}
\definecolor{mycolor}{cmyk}{0,0.5,1,0.5}
\definecolor{mymint}{cmyk}{0.93,0,0.75,0}

\newcommand{\E}{\mathbb E}

\newcommand{\R}{\mathbb R}
\renewcommand{\P}{\mathbb P}

\title[]{A Local Faber-Krahn inequality and Applications\\ to Schr\"odinger Equations}
\author[]{Janna Lierl}
\address{Department of Mathematics, University of Connecticut, 341 Mansfield Road, Storrs, CT 06268}
\email{janna.lierl@uconn.edu}

\author[]{Stefan Steinerberger} 
\address{Department of Mathematics, Yale University, 10 Hillhouse Avenue, New Haven, CT 06511}
\email{stefan.steinerberger@yale.edu}

\begin{document}
\begin{abstract}
We prove a local Faber-Krahn inequality for solutions $u$ to the Dirichlet problem for $\Delta + V$ on an arbitrary domain $\Omega$ in $\R^n$.  Suppose a solution $u$ assumes a global maximum at some point $x_0 \in \Omega$ and $u(x_0)>0$. Let $T(x_0)$ be the smallest time at which a Brownian motion, started at $x_0$, has exited the domain $\Omega$ with probability $\ge 1/2$. For nice (e.g., convex) domains, $T(x_0) \asymp d(x_0,\partial\Omega)^2$ but we make no assumption on the geometry of the domain. Our main result is that there exists a ball $B$ of radius $\asymp T(x_0)^{1/2}$ such that 
$$ \| V \|_{L^{\frac{n}{2}, 1}(\Omega \cap B)} \ge c_n > 0, $$
provided that $n \ge 3$. In the case $n = 2$, the above estimate fails and we obtain a substitute result. The Laplacian may be replaced by a uniformly elliptic operator in divergence form.
This result both unifies and strenghtens a series of earlier results.
\end{abstract}


\maketitle

\section{Introduction and Main Result}

\subsection{Faber-Krahn inequalities.} One of the earliest problems in spectral geometry is the relation between the lowest Dirichlet eigenvalue of the Laplacian and the size of the domain. More precisely, 
assume that a domain $\Omega \subset \mathbb{R}^n$ is given and that we have a nontrivial solution $u$ of the Dirichlet problem
\begin{align*} \begin{split}
\Delta u + \lambda u &= 0 \quad \mbox{ in } \Omega,\\
u|_{\partial \Omega} &= 0.
\end{split}
\end{align*}
By a solution we always mean a solution $u \in W^{1,2}_0(\Omega)$ in the distributional sense.  A solution is trivial if $u =0$ a.e., and nontrivial otherwise. 
 The smallest $\lambda$ for which a nontrivial solution $u$ exists, sometimes denoted $\lambda_1(\Omega)$, can be interpreted as the base
frequency of a vibrating membrane in the shape of $\Omega$. The Faber-Krahn inequality states that
$$ \lambda_1(\Omega) \geq \frac{\pi j_{n/2-1,1}^2}{\Gamma(n/2+1)^{2/n}} |\Omega|^{-2/n},$$
with equality if and only if $\Omega$ is a ball, where $\Gamma$ denotes the Gamma function and $j_{n/2-1,1}$ denotes the smallest positive zero of a Bessel function.
 The Faber-Krahn inequality is a  fundamental theorem in spectral geometry and one of the first results that relates the geometry of a domain $\Omega$ to the solvability of  a differential equation.
 
More generally, a similar inequality holds for the Dirichlet problem with potential $V$,
\begin{align} \begin{split} \label{eq:equation}
 \Delta u + Vu &= 0 \quad \mbox{ in } \Omega,\\
 u|_{\partial \Omega} &= 0.
\end{split}
\end{align}
 De Carli \& Hudson \cite{decarli} proved that a the existence of a nontrivial solution $u \in C(\overline{\Omega})$ of \eqref{eq:equation} implies
$$ \|V\|_{L^{\infty}(\Omega)} \geq \frac{\pi j_{n/2-1,1}^2}{\Gamma(n/2+1)^{2/n}} |\Omega|^{-2/n},$$
and equality is attained if $\Omega$ a metric ball and $V \equiv \lambda_1(\Omega)$ is constant. A slight refinement is given by a not very well known result of Barta \cite{barta} which implies that
$$ \|V\|_{L^{\infty}(\Omega)} \geq \lambda_1(\Omega).$$
This result was later put in a more general context by De Carli, Edward, Hudson \& Leckband \cite{decarli2}. A
sample result (see \cite[Theorem 1.2]{decarli2}) is the following: if $n \ge 3$ and $r >  n/2$, then the existence of a nontrivial solution to \eqref{eq:equation} implies that there is a constant $c_n \in (0,\infty)$ such that
$$ |\Omega|^{\frac{2}{n}-\frac{1}{r}} \| V^+ \|_{L^r(\Omega)} \geq  c_n > 0.$$
The paper \cite{decarli2} also discusses the endpoint case $r=n/2$ and establishes that the result is not valid for $n=2$; this is related to the Sobolev embedding failing for $n=2$. At the core of the argument is an elegant combination of the H\"older inequality, the Sobolev inequality, Green's identity, equation \eqref{eq:equation} and H\"older's inequality again and bears repeating \cite[(1.6)]{decarli2}. We denote the H\"older conjugate of $r >  n/2$ by $q < n/(n-2)$. Then
\begin{align*}
\| u\|_{L^{2q}(\Omega)}^2 &\leq c_n |\Omega|^{\frac{2}{n} - \frac{1}{r}} \int_{\Omega}{|\nabla u|^2 dx} = -c_n  |\Omega|^{\frac{2}{n} - \frac{1}{r}} \int_{\Omega}{ u \Delta u~dx} = c_n  |\Omega|^{\frac{2}{n} - \frac{1}{r}} \int_{\Omega}{ u^2 V dx} \\
&\leq c_n  |\Omega|^{\frac{2}{n} - \frac{1}{r}} \int_{\Omega}{ u^2 V^+ dx}\leq c_n \|u\|_{L^{2q}(\Omega)}^2   |\Omega|^{\frac{2}{n} - \frac{1}{r}}  \| V^+\|_{L^r(\Omega)},
\end{align*}
and from this the assertion follows after cancellation (a similar argument was already used in \cite{decarli3}). These type of inequalities fit naturally into a larger family of results that relate properties of
an elliptic equation to an $L^p-$norm ($p$ often related to the dimension of the space), we refer to the Cwikel-Lieb-Rozenblum inequality \cite{cwikel, liebi, rozenblum}, the Alexandrov-Bakelman-Pucci estimate and various Carleman-type estimates appearing in unique continuation (see Jerison \& Kenig \cite{jer} and the example in Wolff \cite{wolff}).

\subsection{Some motivation.} 
Our interest in these problems was motivated by the following heuristic. Let $\Omega \subset \R^2$ be an elongated domain as shown in Fig. \ref{fig:snake} and let $u$ be a nontrivial solution of \eqref{eq:equation}. The results of Faber-Krahn type discussed above show that $\| V \|_{L^r(\Omega)}$ (for $r>1$)
cannot be arbitrarily small and is bounded below by $\|  V \|_{L^r(\Omega)} \gtrsim_{r}|\Omega|^{1-1/r}$. Clearly, this lower bound on $\| V \|_{L^r(\Omega)}$ decays as $|\Omega|$ increases. However, basic intuition and related results (e.g., \cite{biswas, georg, hayman, manas}) suggest that $\| V \|_{L^r(\Omega)}$ should not decay substantially unless the inradius (the radius of the largest ball fully contained in the domain) increases.

\begin{center}
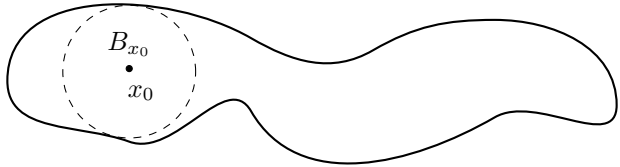
\begin{figure}[h!]
\begin{tikzpicture}[xscale=0.81, yscale=0.81]
\draw [thick] (0,-1) to[out=340, in=120] (2,-0.5) to[out=300, in=210] (6,-0.6)  to[out=30, in=270] (8,-0.4)  to[out=90, in=0] (6,1)   to[out=180, in=30] (4,0.5)   to[out=210, in=330] (2,0.7)    to[out=150, in=90] (-2,0) 
 to[out=270, in=160] (0,-1);
\filldraw (0,0.2) circle (0.05cm);
\node at (0.2, -0.2) {$x_0$};
\draw [dashed] (0,0.15) circle (1.09cm);
\node at (0, 0.6) {$B_{x_0}$};
\end{tikzpicture}
\caption{Long narrow domain and its inradius}
\label{fig:snake}
\end{figure}
\end{center}

This can be seen from various points of view. The geometric perspective is that the solution is bound to have most of its oscillation in the direction orthogonal to the direction
of elongation, which localizes the problem -- in particular, the result should be fairly independent of the transversal direction. A more potential-theoretic perspective is that
the far-field should not act strongly in long narrow domains. Yet another perspective is  to cut the long narrow domain surgically into one that mainly contains the ball. We can then reintroduce Dirichlet boundary condition by only slightly modifying the potential $V$ and without much increase in $\| V \|_{L^r(\Omega)}$. Our main result confirms this intuition.

\subsection{Location of the maximum} 
A different question asks where a nontrivial solution of the Dirichlet problem \eqref{eq:equation} might assume its maximum.
A much simpler but nontrivial question is that of the inradius of the domain $\Omega$. This  question was raised in 1951 by P\'{o}lya \& Szeg\H{o} in their classic book \textit{Isoperimetric Inequalities in Mathematical Physics} \cite{pol}. They raised the question whether there is a constant $c>0$ such that for all simply connected domains $\Omega \subset \mathbb{R}^2$
$$ \mbox{inradius}(\Omega) \geq  c\cdot  \lambda_1(\Omega)^{-1/2}.$$
This inequality was first proven by Makai \cite{makai} in 1965 and, independently, by Hayman \cite{hayman} in 1977. 
No such inequality can hold in higher dimensions because lines cutting through the domain affect the inradius but do not have a strong impact on the lowest eigenvalue.
A celebrated result of Lieb \cite[Corollary 2]{lieb} gives a complete and satisfactory answer.
\begin{thm}[Lieb, 1983] Let $n \geq 3$ and $\Omega \subset \mathbb{R}^n$ be open and non-empty. For every $\varepsilon > 0$, there is a constant $c=c(\varepsilon,n) > 0$ such that
there exists a ball $B$ of radius $c \lambda_1(\Omega)^{-1/2}$ with
$$ |B \cap \Omega| \geq (1- \varepsilon) |B|.$$
\end{thm}
Using an approach of Georgiev \& Mukherjee \cite{georg}, Rachh and the second author \cite{manas} showed that the result also holds if the eigenvalue $\lambda_1(\Omega)$ is replaced by a Schr\"odinger potential. In this case, a ball $B$ of radius 
$$ |B|^{1/n} \sim \|V\|_{L^{\infty}(\Omega)}^{-1/2},$$
centered at a point $x_0$ where the solution $u$ of \eqref{eq:equation} assumes its maximum, has a large intersection with $\Omega$. Biswas \cite{biswas} recently extended the result to fractional Schr\"odinger operators $-(-\Delta)^{\alpha/2} + V$. This line of reasoning was further pursued by Biswas \& L\H{o}rinczi \cite{bis2}.
Put differently, if the maximum is close to the boundary, then the potential has to be large. It is easy to see that this is sharp: consider
$$ u(x,y) = \sin{(n \pi x)} \sin{(m \pi y)} ~\mbox{on}~\Omega=[0,1]^2~\mbox{solving}~-\Delta u + (m^2 + n^2)\pi^2 u = 0,$$
which has global maxima and minima at points whose distance to the boundary is given by
$$ d(\arg \max u, \partial \Omega) \sim \frac{1}{\max(m,n)}  \gtrsim \frac{1}{\sqrt{m^2 + n^2}} = \|V\|_{L^{\infty}(\Omega)}^{-1/2}.$$
We observe that this implies a form of the generalized Faber-Krahn inequality given by De Carli \& Hudson \cite{decarli} (without the sharp constant), since
$$ |\Omega| \gtrsim \mbox{inrad}(\Omega)^n \gtrsim d(\arg \max u, \partial \Omega)^n \gtrsim  \|V\|_{L^{\infty}(\Omega)}^{-n/2},$$
where all the implicit constants depend only on the dimension. Here and henceforth, we use $A \lesssim_{} B$ to denote the existence of a universal constant
such that $A \leq cB$. Writing $A \lesssim_{c_1, c_2, \dots, c_n}  B$ denotes that the constant is allowed to depend on the variables in the subscript and
$A \sim B$ denotes that both $A \lesssim B$ and $B \lesssim A$ hold.

\section{Main results}
\subsection{ Setup.}  An informal summary of the types of results discussed above is the following:
\begin{enumerate}
\item If \eqref{eq:equation} has a solution on a small domain $\Omega$, then $\|V\|_{L^p(\Omega)}$ is large (for a certain allowed range of $p$ depending on the dimension and `large' in the sense that there
exists a lower bound depending on $|\Omega|$).
\item If a solution of \eqref{eq:equation} has a global maximum or minimum that is close to the boundary, then $\|V\|_{L^{\infty}(\Omega)}$ has to be large.
\end{enumerate}
We will prove a result that unifies both these results for general equations of the type
$$ \di(A \cdot \nabla u) + Vu = 0,$$
where we assume that $A=A(x)$ is measurable and there exist constants $0 < \lambda < \Lambda < \infty$ 
\begin{align*} \tag{uniform ellipticity} 
\lambda |\xi|^2 \leq \left\langle A \xi, \xi\right\rangle \leq \Lambda |\xi|^2, \quad \forall \xi \in \R^n.
\end{align*}

The uniform ellipticity constants $\lambda$ and $\Lambda$ impact all subsequent constants; we will suppress this dependence for clarity of exposition. 
For the diffusion process $(X_t)_{t \ge 0}$ generated by the uniformly elliptic operator $ \di(A \cdot \nabla u)$ on $\Omega$ with Dirichlet boundary condition, we let
\begin{align*}
\tau := \inf \{ t>0 : X_t \notin \Omega \} 
\end{align*}
be the first exit time of the domain $\Omega$ (or, alternatively, the first hitting time of the boundary if no boundary condition were imposed).

\begin{definition} \label{def:median exit time} 
We define the {\em median exit time} for the diffusion starting at point $x \in \Omega$ as
\begin{align*}
T_{\eta}(x) := \inf \left\{ t>0 : \P_x(\tau \le t) \ge \eta \right\}.
\end{align*}
\end{definition}
For the purpose of Theorem \ref{thm:1} and Theorem \ref{thm:2}, we make the arbitrary choice $\eta=1/2$ and then drop the subscript $\eta$. However, the implicit constants in our theorems depend on the choice of $\eta$.
\begin{rem} \label{rem:median exit time} Several remarks are in order.
\begin{enumerate} 
\item A similar quantity has already been used in \cite{cheng, jianfeng, manas, stein} under the name of `diffusion distance', in a setting of finite graphs.
\item There exists a constant $0 < c < 1$ depending only on $\lambda, \Lambda$ and the dimension such that  $|B(x_0, r) \cap \Omega| \leq c|B(x_0, r)|$ implies $T(x_0) \leq r^2$.
\item There exists a constant $c> 0$ depending only on $\lambda, \Lambda$ and the dimension such that
 $$\forall x \in \Omega, \qquad T(x) \geq c\inf_{y \in \partial \Omega}{\| x - y\|^2}.$$
This constant is assumed if $\Omega$ is a ball and $x$ is the center.
 \item
If the domain has a finite capacitary  width $w_{\eta'}$ as introduced by Aikawa in \cite{Aik98} then there is a positive constant $c=c(\eta,\eta',n,\lambda,\Lambda)$ such that $T_{\eta}(x) \le c \, w_{\eta'}^2$ for all $x \in \Omega$.
\item If $\Omega \subset \mathbb{R}^2$ is simply connected, then $T(x) \sim_{\lambda, \Lambda} \inf_{y \in \partial \Omega}{\| x - y\|^2}$. In higher dimensions this is true if, e.g., the domain is convex or satisfies an exterior cone condition (see \cite{biswas}). In general, relating the median exit time to any kind of `distance to the boundary' puts restrictions on the geometry of the domain. 
\end{enumerate}
\end{rem}


\subsection{Main result.}

Our main result is the following Faber-Krahn type inequality. It states that if $u$ solves $$\di(A \cdot \nabla u) + Vu = 0,$$ and $|u|$ has its maximum in $x_0 \in \Omega$, then $T(x_0)^{1/2}$ defines a characteristic scale such that the potential $V$ has to be large somewhere inside the domain on that scale. The novelty of this Theorem is that it is independent of the overall shape of the domain.

\begin{theorem} \label{thm:1}
Let $\Omega \subset \R^n$, $n \geq 3$, be a bounded domain. If $u \not\equiv 0$ is a nontrivial solution of
\begin{align*} \begin{split}
\di(A \cdot \nabla u) + Vu &= 0 \quad \mbox{ in } \Omega,\\
 u|_{\partial \Omega} &= 0
\end{split}
\end{align*}
and $|u|$ assumes a global maximum in $x_0$, then there exists a ball $B \subset \mathbb{R}^n$ of radius $T(x_0)^{1/2}$ such that 
$$ \| V^+ \|_{L^{\frac{n}{2},1}(\Omega \cap B)} \geq c_{n, \lambda, \Lambda},$$
where $c_{n, \lambda, \Lambda} > 0$ depends only on the uniform ellipticity constants $\lambda, \Lambda$ and the dimension.
\end{theorem}

\begin{rem} Several remarks are in order.
\begin{enumerate}  
\item  $L^{\frac{n}{2}, 1}$ is the usual Lorentz space refinement of $L^p$ spaces.
\item The proof yields a slightly stronger result: we can apply this result whenever $u$ assumes global maximum in $x_0$ and $u(x_0) > 0$ and whenever $u$ assumes a global minimum in $x_0$ and $u(x_0) < 0$. Since $u$ is nontrivial, the global maximum of $|u|$ falls into one of these categories.
\item 
We note the following immediate corollary: if  $-\Delta u = \lambda_1(\Omega) u$, then
\begin{align*}
  1 \lesssim \| V \|_{L^{\frac{n}{2},1}(\Omega \cap B(x_0, r))}  &=  \| \lambda_1(\Omega) \|_{L^{\frac{n}{2}}(\Omega \cap B(x_0, r))} \\
&\sim \left( \lambda_1(\Omega)^{\frac{n}{2}} r^{n} \right)^{\frac{2}{n}} \sim \lambda_1(\Omega) T(x_0).\end{align*}
This implies the existence of a ball of radius $\sim T(x_0)^{1/2} \gtrsim \lambda_1^{-1/2}$ that has large parts inside $\Omega$ (cf. the properties of the median exit time in Remark \ref{rem:median exit time}). This result was first established by Lieb \cite{lieb} and refined by Georgiev \& Mukherjee \cite{georg}. 
\end{enumerate}
\end{rem}

\subsection{The case $n=2$.}

 The case $n=2$ is slightly different: Theorem 1 stated in $n=2$ dimensions fails. We illustrate this with an example on the unit disk $\mathbb{D} \subset \mathbb{R}^2$ given by De Carli, Edward, Hudson \& Leckband \cite{decarli2}: define the radial function $u_{\varepsilon}(r)$ by
$$ u(r) = \begin{cases} \frac{1}{2} - \log{\varepsilon} -  \frac{1}{2}\varepsilon^{-2} r^2 \qquad &\mbox{if}~0 \leq r \leq \varepsilon\\
- \log{r} \qquad &\mbox{if}~\varepsilon \leq r \leq 1. \\ \end{cases}$$
Both $u_{\varepsilon}$ and its derivative $u_{\varepsilon}'$ are continuous. We observe that $\Delta u_{\varepsilon} \sim \varepsilon^{-2} 1_{\left\{|x| \leq \varepsilon\right\}} $ and $\|u\|_{L^{\infty}} \sim -\log{\varepsilon}$. Put differently,
we have $\Delta u_{\varepsilon} = V_{\varepsilon} u_{\varepsilon}$ for a potential $V_{\varepsilon}$ satisfying  
$$V_{\varepsilon}(x) \sim \varepsilon^{-2}(-\log{\varepsilon})^{-1} 1_{\left\{|x| \leq \varepsilon\right\}}.$$
Obviously, $\| V_{\varepsilon} \|_{L^1(\mathbb{D})} \sim (-\log{\varepsilon})^{-1}$ is not bounded from below. Hence there cannot be a lower bound on $\| V_{\varepsilon} \|_{L^1(\Omega \cap B)}$ for a ball $B \subset \mathbb{D}$. 
Theorem 2 will show that this type of logarithmic behavior is actually the worst possible case: note that the convolution
$$ |V_{\varepsilon}| *  \left|(\log{|x|}) 1_{\left\{|x| \leq 1\right\}}\right|(0) = \int_{\{|y| \leq \varepsilon\}}{ \frac{ |\log{(|y|)}|}{\log{\left(\frac{1}{\varepsilon}\right)}\varepsilon^2} dy}   \sim  1  $$
and therefore, with an implicit constant that is independent of $\varepsilon$,
$$  |V_{\varepsilon}| *  \left|(\log{|x|}) 1_{\left\{|x| \leq 1\right\}}\right|(0) \gtrsim 1.$$
We will prove that this holds in general.
\begin{theorem}[Local Faber-Krahn inequality, $n=2$] \label{thm:2}
Let $\Omega \subset \R^2$ be a bounded domain. Let $u$ is a nontrivial $(u \neq 0)$ solution of
\begin{align*} \begin{split} 
\di(A \cdot \nabla u) + Vu &= 0 \quad \mbox{ in } \Omega,\\
 u|_{\partial \Omega} &= 0.
\end{split}
\end{align*}
There exists a constant $c >0$ depending only on $\lambda, \Lambda$ such that if $|u|$ assumes a global maximum in $x_0$, then there exists a ball $B=B(x,(c T(x_0))^{1/2}) \subset \mathbb{R}^2$ such that 
$$  \int_{B}{ |V^+(y)|  \log{\left(\frac{c T(x_0)}{|x-y|^2}\right)} dy} \geq c_{\lambda, \Lambda}.$$
\end{theorem}

\subsection{Lieb-type inequalities.} 
One particularly interesting consequence are Lieb-type inequalities that follow from the same approach. Avoiding any notion of `distance to the boundary' of the domain, Lieb's theorem relates the norm of the potential to the (Euclidean) geometry of the domain.


\begin{theorem}[Lieb-type inequality] \label{thm:Lieb-type}
Let $n \geq 3$ and let $\Omega \subset \R^n$ be a bounded domain. Let $u$ be a nontrivial solution of the Dirichlet problem
\begin{align*}
\di(A \cdot \nabla u) + Vu &= 0 \quad \mbox{ in } \Omega,\\
 u|_{\partial \Omega} &= 0.
\end{align*}
Suppose $|u|$ assumes its maximum in $x_0 \in \Omega$. Then there is a constant $C=C(n,\lambda,\Lambda) >0$ such that the following holds: If the potential $V$ is so small that, for some $\eta \in (0,1)$, 
 $\left\| V^+ \right\|_{L^{\frac{n}{2},1}(B)} < C^{-1} \eta$ for all balls $B \subset \R^n$ of radius at most $T_{\eta}(x_0)^{1/2}$, then the point $x_0$ must be so close to the boundary that
$$ |B(x_0,T_{\eta}(x_0)^{1/2}) \cap \Omega| \geq \frac{1-2\eta}{1-\eta} | B(x_0,T_{\eta}(x_0)^{1/2})|.$$
\end{theorem}

\section{Proof of Theorem \ref{thm:1}}
This section is devoted to the proof of Theorem \ref{thm:1}. The argument decouples nicely into several different parts. First, we derive the fundamental inequality \eqref{eq:fund ineq}
that was already used in \cite{manas, stein}. The other subsections develop different types of tools that will allow us to extract the desired information from inequality \eqref{eq:fund ineq}. 
As before, we let $(X_t)_{t \ge 0}$ be the diffusion process on $\Omega$, generated by the uniformly elliptic operator $\di(A \cdot \nabla)$, with absorption at the boundary.  We introduce a cemetery state $\Delta$ and set $V(\Delta)=0$ and $u(\Delta )=0$.

\subsection{A preliminary lower bound}\label{sec:eins}
We consider the solution of \eqref{eq:equation} as a steady-state solution of the parabolic equation
$$ \partial_t u - \left( \di(A \cdot \nabla u) + Vu \right) = 0.$$
By the Feynman-Kac formula, 
$$\forall t\geq 0, \qquad  u(x) =  \E_{x} \left(u(X_t) \exp \left( \int_0^t V(X_s) ds \right)  \right).$$
We may assume that $u$ assumes a global maximum in $x_0$ and $u(x_0) > 0$ (otherwise consider $-u$ and note that $-u$ also solves \eqref{eq:equation}). Recall that $\tau$ is the first exit time from $\Omega$. Then,
\begin{align*}
 u(x_0)
&= \E_{x_0} \left(u(X_t) 1_{\left\{\tau > t\right\}}\exp \left( \int_0^t V(X_s) ds \right)  \right) \\
 &\leq u(x_0)  \E_{x_0}\left( 1_{\left\{\tau > t\right\}}  \exp \left( \int_0^t V^+(X_s) ds \right) \right).
 \end{align*}
Since $u(x_0) > 0$, this simplifies to
\begin{align} \label{eq:fund ineq}
  \E_{x_0}\left(   1_{\left\{\tau > t\right\}}  \exp \left( \int_0^t V^+(X_s) ds \right)\right) \geq 1.
\end{align}

\subsection{Khasminskii's Lemma}\label{sec:drei}

\begin{lemma}[Khasminskii's lemma] Let $V \geq 0$ be a measurable function and $(X_s)_{s\ge0}$ be a Markov process on $\mathbb{R}^{n}$ with the property that for some $t>0$ and $\alpha < 1$,
$$ \sup_{x \in \mathbb{R}^{n}}{ \E_{x} \left[ \int_0^t{V(X_s)} ds \right] }= \alpha.$$
Then
$$ \sup_{x \in \mathbb{R}^n}{ \E_{x} \left[ \exp \left( \int_{0}^{t}{V(X_s) ds} \right) \right]} \leq \frac{1}{1-\alpha}.$$
\end{lemma}

Khasminskii's lemma is a classical tool in connection with the Feynman-Kac formula. For the convenience of the reader, we repeat the proof given in \cite{feyn, sim}.
\begin{proof} The argument proceeds by
showing the stronger result
$$ \mathbb{E}_x \left[\frac{1}{m!} \left( \int_0^t V(B_s) ds \right)^m \right] \leq \alpha^m$$
for all non-negative integers $m$. From this, the desired result then follows by summation. Expanding the power allows us to rewrite the statement as 
$$ \mathbb{E}_x \left[ \frac{1}{m!} \int_{0}^{t} \int_{0}^{t} \dots \int_{0}^{t} V(B_{s_1}) V(B_{s_2}) \dots V(B_{s_m}) ds_1 \dots ds_m \right] \leq \alpha^m.$$
There are $n!$ ways of ordering an $n-$tuple of point; therefore, defining
$$ \Delta_m = \left\{(s_1, \dots, s_n): 0 \leq s_1 \leq s_2 \dots \leq s_m  \leq t \right\},$$
we have the equivalent statement
$$  \mathbb{E}_x \left[\int_{\Delta_m} V(B_{s_1}) V(B_{s_2}) \dots V(B_{s_m}) ds_1 \dots ds_m \right] \leq \alpha^m.$$
This is where the Markovian property enters: for any fixed $s_1 \leq s_2 \leq \dots \leq s_{m-1}$
\begin{align*}
&\quad V(B_{s_1}) V(B_{s_2}) \dots V(B_{s_{m-1}})\mathbb{E}_x \left[  \int_{s_{n-1}}^{t}{V(B(s_m)) ds_m} \right]  \\
&\leq V(B_{s_1}) V(B_{s_2}) \dots V(B_{s_{m-1}}) \sup_{y\in \R^n}\mathbb{E}_y  \left[\int_{0}^{t}{V(B(s_m)) ds_m}\right]\\
&\leq V(B_{s_1}) V(B_{s_2}) \dots V(B_{s_{m-1}}) \alpha,
\end{align*}
where we have used that $V \geq 0$. The lemma now follows by induction.
\end{proof}

\subsection{A technical estimate.} \label{sec:vier}
The purpose of this subsection is to provide a self-contained proof of the following lemma for the convenience of the reader. Stronger results could be obtained by appealing to the literature centered around
special functions (especially results dealing with incomplete Gamma functions) but are not needed here.
\begin{lemma} \label{lem:gauss} Let $n \in \mathbb{N}$ and $d>0$. Let $x \in \R^n$. If $n=2$, then
$$ \int_{0}^{d}{ \frac{c_1}{s} \exp \left( -\frac{|x|^2}{c_2 s} \right)}ds \lesssim_{c_1, c_2} \left(1 + \max\left\{0, -\log{\left( \frac{|x|^2}{c_2 d} \right)} \right\} \right) \exp\left(-\frac{|x|^2}{c_2 d}\right).$$
If $n \in \left\{3,4\right\}$, then
$$ \int_{0}^{d}{ \frac{c_1}{s^{n/2}} \exp \left( -\frac{|x|^2}{c_2 s} \right)}ds \lesssim_{c_1, c_2,n }   |x|^{2-n} \exp \left( -\frac{|x|^2}{c_2 d} \right).$$
If $n \geq 5$, then there exists a polynomial $q(\cdot)$ of degree at most $(n-2)/2$ such that
$$ \int_{0}^{d}{ \frac{c_1}{s^{n/2}} \exp \left( -\frac{|x|^2}{c_2 s} \right)}ds \lesssim_{c_1, c_2,n }   |x|^{2-n} q \left( \frac{|x|^2}{c_2 d} \right)\exp \left( -\frac{|x|^2}{c_2 d} \right).$$
\end{lemma}

\begin{proof} The substitution $z= s/|x|^2$ shows 
$$  \int_{0}^{d}{ \frac{c_1}{s^{n/2}} \exp \left( -\frac{|x|^2}{c_2 s} \right)}ds \lesssim_{c_1}  |x|^{2-n}\int_{0}^{d/|x|^2}{ \frac{\exp(-1/(c_2 z))}{z^{n/2}} dz}.$$
Another substitution ($y=1/(c_2 z)$) yields
$$  \int_{0}^{d/|x|^2}{ \frac{\exp(-1/(c_2 z))}{z^{n/2}} dz} \lesssim_{c_2,n}  \int_{|x|^2/(c_2 d)}^{\infty}{y^{\frac{n-4}{2}} e^{-y} dy}.$$ 
We first consider the case $n=2$. If $|x|^2/(c_2 d) \leq 1$ we have
 $$ \int_{|x|^2/(c_2 d)}^{\infty}{y^{-1} e^{-y} dy} \lesssim 1 + \int_{|x|^2/(c_2 d)}^{1}{y^{-1} e^{-y} dy} \lesssim 1  + \int_{|x|^2/(c_2 d)}^{1}{y^{-1} dy} \lesssim 1 - \log{ \left(\frac{|x|^2}{c_2 d} \right)}, $$
and if $|x|^2/(c_2 d) \geq 1$ we have
 $$ \int_{|x|^2/(c_2 d)}^{\infty}{y^{-1} e^{-y} dy} \leq \frac{c_2 d}{|x|^2} \int_{|x|^2/(c_2 d)}^{\infty}{e^{-y} dy} =  \frac{c_2 d}{|x|^2} \exp\left(-\frac{|x|^2}{c_2 d}\right) \leq \exp\left(-\frac{|x|^2}{c_2 d}\right).$$ 
Summarizing, this establishes
 $$ \int_{|x|^2/(c_2 d)}^{\infty}{\frac{1}{y} e^{-y} dy} \lesssim \left(1 + \max\left\{0, -\log{\left( \frac{|x|^2}{c_2 d} \right)} \right\} \right) \exp\left(-\frac{|x|^2}{c_2 d}\right),$$
which is the desired statement for $n=2$. 
The cases $n \in \left\{3,4\right\}$ are even simpler: If $|x|^2/(c_2 d) \leq 1$ then everything is on scale $\sim 1$,
\begin{align*}
\exp\left(-\frac{|x|^2}{c_2 d}\right) \sim 1 \sim   \int_{0}^{\infty}{y^{\frac{n-4}{2}} e^{-y} dy}  \sim  \int_{|x|^2/(c_2 d)}^{\infty}{y^{\frac{n-4}{2}} e^{-y} dy} .
\end{align*}
If $|x|^2/(c_2 d) \geq 1$ then we may use $y^{\frac{n-4}{2}} \leq 1$ for $y \geq 1$ to estimate
$$  \int_{|x|^2/(c_2 d)}^{\infty}{y^{\frac{n-4}{2}} e^{-y} dy} \leq  \int_{|x|^2/(c_2 d)}^{\infty}{e^{-y} dy} = \exp\left(-\frac{|x|^2}{c_2 d}\right).$$
If $n \geq 5$, then we can bound $y^{\frac{n-4}{2}} \lesssim 1 + y^k$, where $k$ is the smallest integer bigger or equal than $(n-4)/2$. It then suffices to remark that integration by parts implies, for $k \in \mathbb{N}$,
$$ \int_{|x|^2/(c_2 d)}^{\infty}{y^{k} e^{-y} dy} =  \left(\frac{|x|^2}{c_2 d} \right)^k \exp\left(-\frac{|x|^2}{c_2 d}\right) +  \int_{|x|^2/(c_2 d)}^{\infty}{y^{k-1} e^{-y} dy}$$
and an iterative application implies the assertion. 
\end{proof}

\subsection{An upper bound on a convolution.} \label{sec:funf}
It is well known (see the work of Aronson \cite{aronson}) that the uniformly elliptic operator $\mbox{div}(A \cdot \nabla)$ admits a heat kernel $p_t(x,y)$ satisfying the Gaussian upper bound
$$ p_s(x,y) \leq \frac{c_1}{s^{n/2}} \exp \left( -\frac{|x-y|^2}{c_2 s} \right), \quad \forall s>0, x,y \in \R^n,$$
for some $c_1, c_2 > 0$ depending only on $\lambda,\Lambda,n$. The last ingredient for the proof of Theorem \ref{thm:1} is the following estimate.

\begin{lemma}\label{lem:est} Let $n \geq 3$, let $f: \mathbb{R}^n \rightarrow [0,\infty)$. For every $d > 0$,
$$ \sup_{x \in \mathbb{R}^n} \int_{\mathbb{R}^n}{ f(y) \int_{0}^{d^2}{p_s(x,y) ds } \, dy} \lesssim_{c_1,c_2,n} \sup_{|B| \leq  d^n} \| f\|_{L^{\frac{n}{2},1}(B)},$$
where the supremum ranges over all balls $B\subset \R^n$ with volume $|B| \leq d^n$.
\end{lemma}

\begin{proof} We fix $x \in \mathbb{R}^n$ and let $B_1 = B(x,d)$. We choose countably many balls $B_i=B(x_i,d)$, $i \ge 2$, in such a way that every point in $\mathbb{R}^n$ is contained in at most $N$ of these balls, where $N$ depends only on the dimension $n$.
Then
$$ \int_{\mathbb{R}^n}{ f(y) \int_{0}^{d^2}{p_s(x,y) ds} \, dy} \leq \sum_{i=1}^{\infty} \int_{B_i}{ f(y) \int_{0}^{d^2}{p_s(x,y) ds} \, dy }.$$
By Lemma \ref{lem:gauss},
$$ \int_{B_i} f(y) \int_{0}^{d^2}{p_s(x,y) ds} \, dy \lesssim_{c_1,c_2}  \int_{B_i} \frac{f(y) }{|x-y|^{n-2}}  q \left( \frac{|x-y|^2}{c_2 d^2}\right) \exp\left(-\frac{|x-y|^2}{c_2 d^2}\right) dy$$
We note that there is one term that becomes singular on the diagonal $x=y$ while the other terms all exhibit decay.
The refined H\"older inequality due to O'Neil \cite{oneil},
 $$ \| fg \|_{L^1(\mathbb{R}^n)} \lesssim_n \|f\|_{L^{\frac{n}{2}, 1}(\mathbb{R}^n)} \| g\|_{L^{\frac{n}{n-2}, \infty}(\mathbb{R}^n)},$$
together with the fact that 
$$ \frac{1}{|x-y|^{n-2}} \in L^{\frac{n}{n-2},\infty}(\mathbb{R}^n,dy),$$
implies that
$$  \int_{B_i}{ \frac{ f(y) }{|x-y|^{n-2}} dy} \lesssim  \| f \|_{L^{\frac{n}{2}, 1}(B_i)}$$
for each $i \ge 1$. We apply the above estimate to each ball $B_i$ that is at most distance $d/2$ away from $x$. The number of such $B_i$'s can be bounded in terms on $n$ only. 

It  remains to consider those balls $B_i$ whose distance to $x$ is at least $d/2$. A simple counting estimate shows that
we have roughly $\sim \ell^{n-1}$ balls at distance $\sim \ell \cdot d$ from $x$. If $B_i$ is such a ball then
$$  \frac{1 }{|x-y|^{n-2}}  q\left( \frac{|x-y|^2}{c_2 d^2} \right)\exp\left(-\frac{|x-y|^2}{c_2 d^2}\right) \lesssim  \frac{ \ell^{2 \cdot \deg q} }{d^{n-2}\ell^{n-2}}   \exp\left( - \frac{ \ell^2}{c_2}\right).$$
Using this asymptotic, we will bound the integral
$$ I =   \int_{B_i} \frac{f(y) }{|x-y|^{n-2}}   q\left( \frac{|x-y|^2}{c_2 d^2} \right) \exp\left(-\frac{|x-y|^2}{c_2 d^2}\right) dy.$$
Applying the usual H\"older inequality with $p= n/2$ and $p' = n/(n-2)$, we get
$$ I \lesssim  \ell^{2 \cdot \deg q -n + 2}  \exp\left( - \frac{ \ell^2}{c_2}\right) \sup_{|B| \leq d^n} \| f\|_{L^{\frac{n}{2}}(B)}.$$
Altogether, we obtain
\begin{align*}
 \sup_{x \in \mathbb{R}^n} \int_{\mathbb{R}^n}{ f(y) \int_{0}^{d^2}{p_s(x,y) ds} \, dy} &\lesssim  \sup_{|B| \leq d^n} \| f\|_{L^{\frac{n}{2},1}(B)}  \\
&+   \left( \sum_{\ell=1}^{\infty}  \ell^{2 \cdot \deg q -n + 2}   \ell^{n-1}  \exp\left( - \frac{ \ell^2}{c_2}\right) \right)  \sup_{|B| \leq  d^n} \| f\|_{L^{\frac{n}{2}}(B)}\\
&\lesssim   \sup_{|B| \leq d^n} \| f\|_{L^{\frac{n}{2},1}(B)}.
\end{align*}
\end{proof}

\subsection{Proof of Theorem \ref{thm:1}.}
Given a point $x_0$ at which the solution attains its maximum, estimate \eqref{eq:fund ineq} together with the Cauchy-Schwarz inequality imply that
\begin{align*}
 1 \leq  \E_{x_0}\left(   1_{\left\{\tau > t\right\}}  \exp \left( \int_0^t V^+(X_s) ds \right) \right) \leq \mathbb{P}\left( \tau > t \right)^{1/2} \left(\mathbb{E}_{x_0} \exp \left( \int_0^t 2V^+(X_s) ds \right)\right)^{1/2}
\end{align*}
for all $t > 0$.
We choose $t = T(x_0)$ to be the median exit time $T(x_0)$ that we introduced in Definition \ref{def:median exit time}. Then
$$ \mathbb{P}_{x_0}\left( \tau > T(x_0) \right)^{1/2} \le \frac{1}{\sqrt{2}},$$
and therefore
$$   \mathbb{E}_{x_0} \left[ \exp \left( \int_0^{T(x_0)} 2V^+(X_s) ds \right) \right] \geq 2.$$
Khasminskii's Lemma, discussed in Section \ref{sec:drei}, now implies
$$  \sup_{x \in \mathbb{R}^n}{ \E_{x} \left[ \int_0^{T(x_0)}{ 2V^+(X_s)} ds \right] } \geq \frac{1}{2}.$$
Finally, by Lemma \ref{lem:est},
$$ \frac{1}{2} \leq \sup_{x \in \mathbb{R}^n} \E_{x} \left[ \int_0^{T(x_0)}{2V^+(X_s)} ds \right]   \lesssim_{n, \lambda, \Lambda}   \sup_{|B| \leq  T(x_0)^{n/2}} \left\| V^+\right\|_{L^{\frac{n}{2},1}(B)}.$$
Recall that $V \equiv 0 $ outside $\Omega$. The proof of Theorem \ref{thm:1} is now complete.

\section{Proofs of Theorem \ref{thm:2} and \ref{thm:Lieb-type}}

\subsection{Proof of Theorem \ref{thm:2}.}  The proof follows the same line of thought as the proof of Theorem 1, the only difference is one technical lemma which we provide here.
\begin{lemma} \label{lem:final} Let $d > 0$ and assume $g: \mathbb{R}^2 \rightarrow (0,\infty)$ satisfies
$$ g(x) \lesssim \left(1+ \log{\frac{d^2}{|x|^2}} 1_{\{|x| \leq d\}} \right) \exp\left(-\frac{|x|^2}{ d^2} \right).$$
Then
$$ \sup_{x \in \mathbb{R}^2} \int_{\mathbb{R}^2}{f(x-y) g(y) dy} \lesssim \sup_{x \in \mathbb{R}^2}  \int_{\{|y| \leq d\}}{ |f(x-y)|  \log{\left(\frac{d^2}{|y|^2}\right)} dy}.$$
\end{lemma}
\begin{proof} We may assume that the supremum is assumed in the origin (after possibly translating the function). The desired inequality then reads
$$  \int_{\mathbb{R}^2}{f(y) g(y) dy} \lesssim \sup_{x \in \mathbb{R}^2}  \int_{\{|y| \leq d\}}{ |f(x-y)|  \log{\left(\frac{d^2}{|y|^2}\right)} dy}.$$
As in the proof of Lemma 6, we cover $\mathbb{R}^2$ by disks $D_i=B(x_i,d)$, $i \ge 1$, of radius $d$ in such a way that $x_1=0$, each
point in $\mathbb{R}^2$ is at most distance $d/10$ away from some $x_i$, and each point in $\mathbb{R}^2$ is contained in at most $c$ of these disks
(a simple lattice construction shows this to be possible; one could ask for constructions that minimize $c$ and questions of these types have been studied independently, see F\"uredi \& Loeb \cite{furedi} - this is, of course, not required here). We can now bound
\begin{align*}
 \int_{\mathbb{R}^2}{f(y) g(y) dy}   &\lesssim   \int_{\mathbb{R}^2} f(y)  \left(1+ \log{\frac{d^2}{|y|^2}} 1_{\{|y| \leq d\}} \right)   \exp\left(-\frac{|y|^2}{ d^2} \right) dy  \\
&\lesssim \int_{\{|y| \leq d\}}{|f(y)|  \log{\left(\frac{d^2}{|y|^2}\right)} dy} + \int_{\mathbb{R}^2}{|f(y)|   \exp\left(-\frac{|y|^2}{ d^2} \right) dy}.
\end{align*}
The first term is easy to bound since, trivially,
$$ \int_{\{|y| \leq d\}}{|f(y)|  \log{\left(\frac{d^2}{|y|^2}\right)} dy} \leq \sup_{x \in \mathbb{R}^2} \int_{\{|y| \leq d\}}{|f(x-y)|  \log{\left(\frac{d^2}{|y|^2}\right)} dy}.$$
We now deal with the second term: Clearly,
\begin{align*}
\int_{\mathbb{R}^2}{|f(y)|   \exp\left(-\frac{|y|^2}{ d^2} \right) dy} \leq \sum_{i=1}^{\infty} \int_{D_i}{|f(y)|   \exp\left(-\frac{|y|^2}{ d^2} \right) dy}.
\end{align*}
We may assume that the disks $i$ are ordered in increasing distance from the origin so that 
$$ d(D_i, 0) = \inf_{x \in D_i}|x| \geq d \sqrt{c_3 i}$$
for some $c_3>0$ and for sufficiently large $i$. This implies
$$ \int_{D_i}{|f(y)|   \exp\left(-\frac{|y|^2}{ d^2} \right) dy} \lesssim  \exp\left(-c_3 \cdot i \right)    \int_{D_i}{|f(y)|  dy}.$$

This leads to summable decay, it now suffices to show the uniform estimate
$$ \sup_{i \in \mathbb{N}}  \int_{D_i}{|f(y)|  dy} \lesssim \sup_{x \in \mathbb{R}^2} \int_{\{|y| \leq d\}}{|f(x-y)|  \log{\left(\frac{d^2}{|y|^2}\right)} dy}.$$
This seems a bit tricky at first (because the logarithm vanishes at $|y| = d$) but is easily compensated by the fact that for every point in $x \in \mathbb{R}^2$, there
exists a disk $D_j$ whose center $x_j$ is at most distance $d/10$ away. We fix an arbitrary $i$ and let 
$$A = \left\{j \in \mathbb{N}: |x_j - x_i| \leq d \right\}.$$
This set is finite and its cardinality only depends on $c$ (this could be made explicit by fixing a sufficiently fine lattice but this is not required). We now claim that
for all $y \in D_i$
$$ |f(y)| \leq \sum_{a \in A}{|f( x_j-(x_j-y) )|  \log{\left(\frac{d^2}{|x_j - y|^2}\right)}}.$$
This is easy to see: for every $y$ there exists $x_j$ with $|x_j - y| \leq d/10$ which ensures that at least one logarithmic factor is bigger than $\log{(100)} \geq 1$. Therefore, with a change of variables,
\begin{align*}
 \int_{D_i}{ |f(y)|  dy} &\leq \sum_{j \in A}{ \int_{D_a}{|f(z)|   \log{\left(\frac{d^2}{|z-x_j|^2}\right)} dz}}\\
&= \sum_{j \in A}{ \int_{\left\{|z| \leq d\right\}}{|f(z+x_j)|   \log{\left(\frac{d^2}{|z|^2}\right)} dz}}\\
&\leq (\#A) \sup_{x \in \mathbb{R}^2} \int_{\left\{|z| \leq d\right\}}{|f(x - z)|   \log{\left(\frac{d^2}{|z|^2}\right)} dz} \\
&\lesssim_{c}  \sup_{x \in \mathbb{R}^2} \int_{\left\{|z| \leq d\right\}}{|f(x - z)|   \log{\left(\frac{d^2}{|z|^2}\right)} dz}.
\end{align*}
\end{proof}

\begin{proof}[Proof of Theorem 2.] Arguing exactly as in the proof of Theorem 1, we arrive at
$$ \frac{1}{2} \leq \sup_{x \in \R^2} \E_{x} \left[ \int_0^{T(x_0)}{2V^+(X_s)} ds \right].$$
Interchanging the order of integration and using Lemma 5 yields
$$  \E_{x} \left[ \int_0^{T(x_0)}{2V^+(X_s)} ds \right] \lesssim \int_{\mathbb{R}^2}{  V^{+}(y-x)\left(1 + \max\left\{0, \log{\left( \frac{c_2 T(x_0)}{|y|^2} \right)} \right\} \right) \exp\left(-\frac{|y|^2}{c_2 T(x_0)}\right) dy}.$$
The argument concludes by using Lemma 7 to bound the supremum via
$$ \sup_{x \in \R^2} \E_{x} \left[ \int_0^{T(x_0)}{2V^+(X_s)} ds \right] \lesssim \sup_{x \in \mathbb{R}^2}  \int_{\{|y| \leq (c_2 T(x_0))^{1/2}\}}{ |V^+(x-y)|  \log{\left(\frac{c_2 T(x_0)}{|y|^2}\right)} dy},$$
which is the desired statement.
\end{proof}

\subsection{Proof of Theorem 3.}

\begin{proof}  Let $\eta \in (0,1)$. Let $x_0 \in \Omega$ be a point in which $|u|$ assumes its maximum. Suppose
$$ |B(x_0,T_{\eta}(x_0)^{1/2}) \cap \Omega| < \frac{1-2\eta}{1-\eta} | B(x_0,T_{\eta}(x_0)^{1/2})|.$$
By Definition \ref{def:median exit time} of the median exit time $T_{\eta}(x_0)$, we have
\begin{align*}
\P_{x_0}( \tau > T_{\eta}(x_0)) < 1 - \eta.
\end{align*}
This, estimate \eqref{eq:fund ineq}, and the Cauchy-Schwarz inequality imply that
\begin{align*} 
1 &\leq  \E_{x_0}\left[ 1_{\left\{\tau > T_{\eta}(x_0) \right\}}  \exp \left( \int_0^{T_{\eta}(x_0)} V^+(X_s) ds \right) \right] \\
&\leq \mathbb{P}\left( \tau > T_{\eta}(x_0) \right)^{1/2} \left( \mathbb{E}_{x_0} \left[ \exp \left( \int_0^{T_{\eta}(x_0)} 2V^+(X_s) ds \right)\right] \right)^{1/2} \\
&\leq \left(1 - \eta \right)^{1/2}  \left( \mathbb{E}_{x_0} \left[ \exp \left( \int_0^{T_{\eta}(x_0)} 2V^+(X_s) ds \right) \right] \right)^{1/2}.
\end{align*}
Thus,
$$  \E_{x_0} \left[ \exp \left( \int_{0}^{T_{\eta}(x_0)}{2V^{+}(X_s) ds} \right) \right] \geq \frac{1}{1-\eta}.$$
By Khasminskii's Lemma, 
$$ \sup_{x \in \R^n}   \E_{x}    \int_{0}^{T_{\eta}(x_0)}{2V^{+}(X_s) ds} \ge \eta.$$
Lemma \ref{lem:est} then implies that there is a constant $C=C(n,\lambda,\Lambda)$ such that
$$ \eta \leq \sup_{x \in \R^n}   \E_{x}  \int_{0}^{T_{\eta}(x_0)}{2V^+(X_s) ds}  \le C \sup_{|B| \leq  T_{\eta}(x_0)^{n/2}} \left\| V\right\|_{L^{\frac{n}{2},1}(B)},$$
where the supremum ranges over all balls $B$ of volume at most $T_{\eta}(x_0)^{n/2}$. If, however, $\left\| V^+ \right\|_{L^{\frac{n}{2},1}(B)} < C^{-1} \eta$ for all balls $B$ of volume at most $T_{\eta}(x_0)^{n/2}$, then we have a contradiction, and therefore
$$ |B(x_0,T_{\eta}(x_0)^{1/2}) \cap \Omega| \geq \frac{1-2\eta}{1-\eta} | B(x_0,T_{\eta}(x_0)^{1/2})|.$$
\end{proof}

\bibliographystyle{siam}
\bibliography{bibliography}

\end{document}